\documentclass[11pt]{article}
\usepackage{graphicx,verbatim}
\usepackage{lineno}
\usepackage{lipsum,pdflscape}
\usepackage{amsmath,amssymb}
\usepackage{xcolor,colortbl}
\usepackage{xspace}
\usepackage{color}
\usepackage{epsfig}
\usepackage[algosection,lined,ruled,linesnumbered]{algorithm2e}
\usepackage{subfigure}

\graphicspath{{./Fig/}}
\newtheorem{theorem}{Theorem}[section]

\newtheorem{corollary}[theorem]{Corollary}
\newtheorem{lemma}[theorem]{Lemma}

\newtheorem{definition}[theorem]{Definition}
\newtheorem{example}[theorem]{Example}

\newcommand{\qed}{\hfill $\square$\medskip}

\begin{document}

\title{ Domination polynomial and total domination polynomial of zero-divisor graphs of commutative rings}

\author{
Saeid Alikhani$^{}$\footnote{Corresponding author}
\and
Fatemeh Aghaei
}

\date{\today}

\maketitle

\begin{center}
Department of Mathematical Sciences, Yazd University, 89195-741, Yazd, Iran\\
{\tt alikhani@yazd.ac.ir, aghaeefatemeh29@gmail.com}
\end{center}


\begin{abstract}
 The domination polynomial (the total domination polynomial) of a graph $ G $ of order $ n $ is the generating function of  the number of  dominating sets (total dominating sets) of $ G $ of any size.   
 In this paper, we study the domination polynomial and the total domination polynomial  of zero-divisor graphs of the ring $ \mathbb{Z}_n $ where $ n\in\lbrace 2p, p^2, pq, p^2q, pqr, p^\alpha\rbrace $, and $ p, q, r $ are primes with $ p>q>r>2 $. 

\end{abstract}

\medskip
\noindent{\bf AMS Subj.\ Class.}: 05C69, 05C25.

\section{Introduction}
Let $ G $ be a simple graph. For any vertex $ v\in V(G) $, the open neighborhood of $ v $ is the set $ N(v)=\lbrace u\in V(G)\vert u\sim v\rbrace $ and the closed neighborhood is the set $ N[v]=N(v)\cup\lbrace v\rbrace$. For a set $ S\subseteq V $, the open neighborhood of $ S $ is $ N(S)=\bigcup_{v\in S}N(v) $ and the closed neighborhood of $ S $ is $ N[S]=N(S)\cup S $. A set $ S\subseteq V $ is a dominating set if $ N[S]=V $, or equivalently, every vertex in $ V\setminus S $ is adjacent to at least one vertex in $ S $. The domination number $ \gamma(G) $ is the minimum cardinality of the dominating set in $ G $. The total dominating set is a subset $ D $ of $ V $ that every vertex of $ V $ is adjacent to some vertices of $ D $. The total domination number of $ G $ is equal to minimum cardinality of total dominating set in $ G $ and denoted by $ \gamma_{t}(G) $. An $ i $-subset of $ V(G) $ is a subset of $ V(G) $ of cardinality $ i $. Let $ \mathcal{D}(G,i) $ be the family of dominating sets of $ G $ which are $ i $-subsets and let 
$ d(G,i)=\vert \mathcal{D}(G,i)\vert $.
The polynomial $$ D(G,x)=\sum_{i=0}^{n}d(G,i)x^{i}, $$ is defined as domination polynomial of $ G $  (\cite{1,2}). Likewise, the total domination polynomial of $ G $ is 
$$ D_t(G,x)=\sum_{i=0}^{n}d_t(G,i)x^{i},$$ where $ d_t(G,i) $ is the number of total dominating sets of $ G $ of size $ i $ (\cite{6}).

The domination polynomial, the total domination polynomial and their roots were well-studied in literature, see e.g., \cite{Brown}. Also the classification of graphs based dominating equivalence is an interesting problem and many graphs determine by their domination polynomials  \cite{complete}.

 Similar to the domination polynomial, the independent domination polynomial of a graph $G$ is a generating function of number of the independent dominating sets of $G$ \cite{7}. The simplicial complexes arising from zero-divisor graphs has been studied in \cite{Demeyer}. The independent sets of 
 a graph is correspond to the faces of its independence complex, and so 
 the independent and independent domination polynomials of
 a graph provide a face enumeration for its simplicial complex \cite{Klee}.
 
Zero-divisor graph of a commutative ring was introduced in the work of Beck in \cite{Beck}. Beck was interested in coloring of rings and the vertex set of graph consists of the all elements of the ring in his definition. Later, the definition of zero-divisor graph of a commutative ring has been modified by Anderson and Livingston \cite{Anderson}. They defined the zero-divisor graph of a commutative ring on nonzero zero-divisor elements of the ring.

 In recent years, the study of zero divisor graphs has grown in various directions. Actually, it is the interplay between the ring theoretic properties of a ring $R$ and the
graph theoretic properties of its zero divisor graph \cite{2n,Anderson}. There are many papers which studied some parameters and topological indices of the zero-divisor graphs. For example we refer reader to see e.g., \cite{Ahmadi,Asir,Lucas}.
 Recently, G\"ursoy, Ulker and G\"ursoy in \cite{Gursoy} have studied  the independent domination polynomials of some zero-divisor graphs of the rings.

This paper is organized as follows. In Section 2, we give some notions about zero-divisor graphs. In Section 3, we investigate the  domination (total) polynomial of some zero-divisor graphs of the rings. We suppose that $ p, q, $ and $ r $ are distinct prime numbers such that $ p>q>r>2 $.

\section{Preliminaries}
This section is devoted to the definitions and notions which will be needed  for the rest of the paper. In the following, we give the definition of the zero-divisor graph.

\begin{definition}\rm{\cite{Anderson}}
 Let $ \mathbb{Z}_n $ be the ring of integers modulo $ n $. The zero-divisor graph $ \Gamma(\mathbb{Z}_n) $ is the simple undirected graph without loops which has its vertex set coincides with the nonzero zero-divisors of $ \mathbb{Z}_n $ and two distinct vertices $ u $ and $ v $ in $ \Gamma(\mathbb{Z}_n) $ are adjacent whenever $ uv=0 $ in $ \mathbb{Z}_n $.
\end{definition}

\begin{example} {\rm \cite{Gursoy}}
	For the graph $\Gamma(\mathbb{Z}_{75})$,  we have $|V(\Gamma(\mathbb{Z}_{75}))|= 34$ and $|E(\Gamma(\mathbb{Z}_{75}))| = 86$. This graph  has shown in Figure \ref{gn}. 
\end{example} 
\begin{figure}[ht]
	\centering
	\includegraphics[scale=0.6]{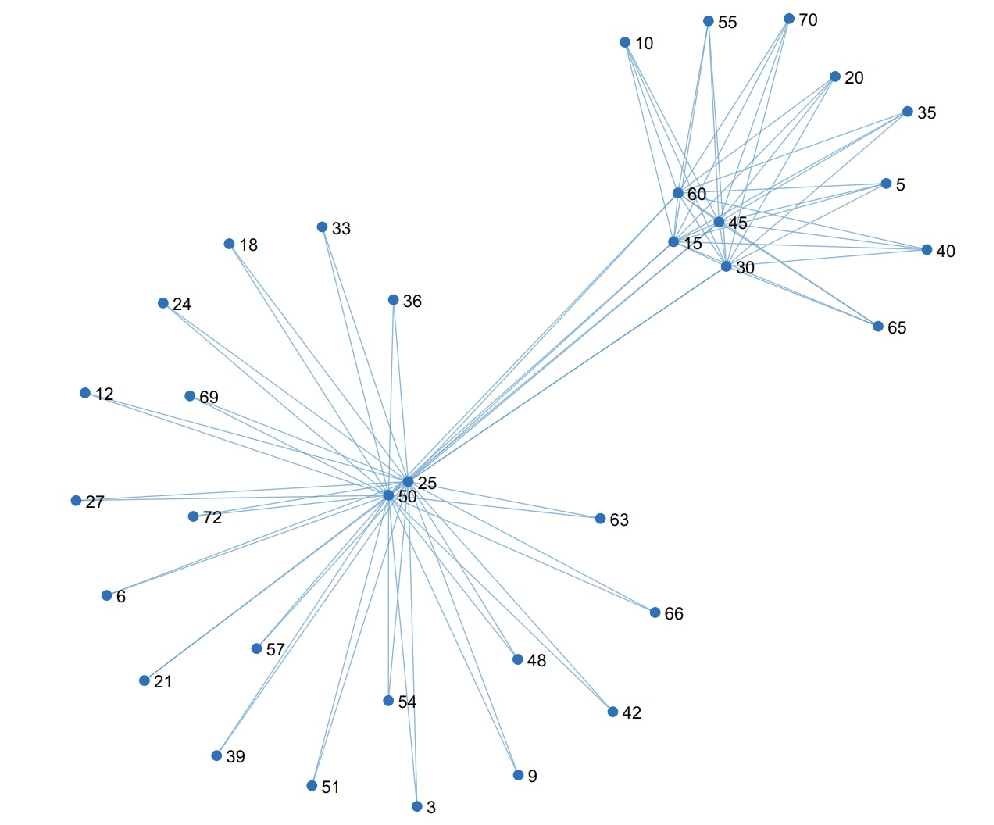}
	\caption{Zero-divisor graph $\Gamma(\mathbb{Z}_{75})$, see \cite{Gursoy}. } \label{gn} 
\end{figure} 

An integer $ d $ is called a \textit{proper divisor} of $ n $ if $ 1<d<n $ and $ d\mid n $. Let $ d_1,...,d_k $ be the distinct proper divisors of $ n $. For $ 1\leqslant i\leqslant k $, consider  the following sets:
\[
V_{d_i}=\lbrace x\in\mathbb{Z}_n:\gcd(x,n)=d_i\rbrace.
\]
The sets $ V_{d_1},...,V_{d_k} $ are pairwise disjoint and we can partition the vertex set of $ \Gamma(\mathbb{Z}_n) $ as 
\[
V(\Gamma(\mathbb{Z}_n) )= \bigcup_{i=1}^kV_{d_i}.
\]
The following lemma gives the cardinalities of each vertex subset of $ \Gamma(\mathbb{Z}_n) $.
\begin{lemma}\rm{\cite{Young}}
Let $ n $ be a positive integer with distinct divisors $ d_1, d_2,..., d_r $. If $ V_{d_i}=\lbrace x\in\mathbb{Z}_n:\gcd(x,n)=d_i\rbrace $ for $ i = 1, 2,...,r, $ then $ |V_{d_i}|=\phi(\frac{n}{d_i}) $, where $ \phi $ is the Euler's totient function.
\end{lemma}

\begin{lemma}\rm{\cite{Chattopadhyay}}
For $ i,j\in\lbrace 1,...,k\rbrace $, a vertex of $ V_{d_i} $ is adjacent to a vertex of $ V_{d_j} $ in $ \Gamma(\mathbb{Z}_n) $ if and only if $ n $ divides $ d_id_j $. 
\end{lemma}
\begin{corollary}\rm{\cite{Chattopadhyay}}\label{induced}
\begin{itemize}
\item[i)]
For $ i\in\lbrace 1,...,k\rbrace $, the induced subgraph $ \Gamma(V_{d_i}) $ of $ \Gamma(\mathbb{Z}_n) $ on the vertex set $ V_{d_i} $ is either the complete graph $ K_{\phi({n}/{d_i})} $ or its complement graph $ \overline{K}_{\phi({n}/{d_i})} $ . Indeed, $ \Gamma(V_{d_i}) $ is $ K_{\phi({n}/{d_i})} $ if and only if $ n $ divides $ d_i^2 $.
\item[ii)]
For $ i,j\in\lbrace 1,...,k\rbrace $ with $ i\neq j $, a vertex of $ V_{d_i} $ is adjacent to either all or none of the vertices of $ V_{d_j} $ in $ \Gamma(\mathbb{Z}_n) $.
\end{itemize}
\end{corollary}

\section{Main results} 
In this section, we study  domination (total) polynomial of zero-divisor graphs of rings $ \mathbb{Z}_n $, where $ n\in\lbrace 2p, p^2, pq, p^2q, pqr, p^{\alpha}\rbrace $ for distinct prime numbers $ p, q $ and $ r $. 
 
 \begin{theorem}
 For prime $ p $, 
 \begin{enumerate} 
 	\item[(i)] 
 	$ D(\Gamma(\mathbb{Z}_{p^2}),x)=(1+x)^{p-1}-1.$
 	
 	\item[(ii)]  
 	$ D_t(\Gamma(\mathbb{Z}_{p^2}),x)=(1+x)^{p-1}-(p-1)x-1.$
\end{enumerate} 
 \end{theorem}
 \begin{proof}
 The integer $ p $ is only proper divisor of $ p^2 $. By Corollary \ref{induced}, $ \Gamma(\mathbb{Z}_{p^2}) $ is the complete graph $ K_{\phi(p)} $. Since $D(K_n,x)=(1+x)^n-1$, $D_t(K_n,x)=(1+x)^n-nx-1$ and $\phi(p)=p-1$, so  we have the results. \qed
 \end{proof}
 \begin{theorem}
 For prime $ p $, 
 \begin{enumerate} 
 	\item[(i)] $ D(\Gamma(\mathbb{Z}_{2p}),x)=x^{p-1}+x(1+x)^{p-1}.$
  \item[(ii)] 
    $D_t(\Gamma(\mathbb{Z}_{2p}),x)=x[(1+x)^{p-1}-1].$
\end{enumerate} 
 \end{theorem}
 \begin{proof}
 The integers $ 2 $ and $ p $ are the proper divisors of $ 2p $, so the vertex set of this graph can be partitioned into two distinct subsets as $ V_2=\lbrace 2x : x=1,...,p-1\rbrace$  and $ V_p=\lbrace p\rbrace$. 
By Corollary \ref{induced}, $ \Gamma(\mathbb{Z}_{2p}) $ is the star graph $ K_{1,\phi(p)} $. Since $D(K_{1,n},x)=x(1+x)^{n}+x^n$, $D_t(K_{1,n},x)=x((1+x)^n-1)$ and $\phi(p)=p-1$, so
 we have the results. \qed
 \end{proof}
 
We need the following theorem which is the domination polynomial of join of two graphs.  
 The join $G_1 \vee G_2$ of two graph $G_1$ and $G_2$ with disjoint vertex sets $V_1$ and $V_2$ and
 edge sets $E_1$ and $E_2$ is the graph union $G_1\cup G_2$ together with all the edges joining $V_1$ and  $V_2$.

\begin{theorem}\label{dpjoin}{\rm \cite{2}}
	Let $G_1$ and $G_2$ be  graphs of orders $n_1$ and $n_2$,
	respectively. Then
	\[ 
	D(G_1\vee G_2,x) = \Big((1+x)^{n_1}-1\Big)\Big((1+x)^{n_2}-1\Big)+D(G_1,x)+D(G_2,x).
	\] 
\end{theorem}

 \begin{theorem}
 If $ p>q>2 $ are prime numbers, then,
\begin{enumerate} 
\item[(i)] 
 $D(\Gamma(\mathbb{Z}_{pq}),x)=[(1+x)^{\phi(p)}-1][(1+x)^{\phi(q)}-1]+x^{\phi(p)}+x^{\phi(q)}$.
 \item[(ii)]
  $ D_t(\Gamma(\mathbb{Z}_{pq}),x)=[(1+x)^{\phi(p)}-1][(1+x)^{\phi(q)}-1]. $
 \end{enumerate}
 \end{theorem}
 \begin{proof}
 	\begin{enumerate} 
 		\item[(i)] 
  The integers $ p $ and $ q $ are the proper divisors of $ pq $, so the vertex set of this graph can be partitioned into two distinct subsets as
 \[
 V_p=\lbrace px : x=1,...,q-1\rbrace,
 \]
 \[
 V_q=\lbrace qx : x=1,...,p-1\rbrace.
 \]
 Consequently, by Corollary \ref{induced}, $ \Gamma(\mathbb{Z}_{pq}) $ is the complete bipartite graph $ K_{\phi(p),\phi(q)} $. Since $ K_{\phi(p),\phi(q)}=\overline{K}_{\phi(p)}\vee\overline{K}_{\phi(q)} $, so by Theorem \ref{dpjoin}, we have the result. 
 
 \item[(ii)] 
  Since $ \phi(p)>\phi(q)>1 $, so $ \gamma(\Gamma(\mathbb{Z}_{pq}))=\gamma_t(\Gamma(\mathbb{Z}_{pq}))=2 $. Every dominating set of size $ i $ is a total dominating set, if it contains some vertices of both subsets $  V_p $ and $  V_q $. The subsets $  V_p $ and $  V_q $ with cardinality $ \phi(q) $ and $ \phi(p) $ respectively, are dominating sets but not total dominating sets. Therefore
 for $ i\geqslant 2 $,
 \[
 d_t(\Gamma(\mathbb{Z}_{pq}),i)=
 \left\{
 \begin{array}{lr}
 d(\Gamma(\mathbb{Z}_{pq}),i)-1
 \quad\mbox{if ~~$i=q-1 ~~or~~ p-1$,}\\[15pt]
 d(\Gamma(\mathbb{Z}_{pq}),i)
 \quad\mbox{~~~~~otherwise.}\\ 
 \end{array}
 \right.
 \]
 and so the result follows. 
 \qed
 \end{enumerate} 
 \end{proof}

 \begin{figure}[ht]
 	\centering
 	\includegraphics[scale=0.8]{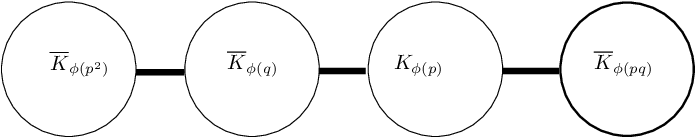}
 	\caption{The graph $ \Gamma(\mathbb{Z}_{p^2q}) $}\label{Z_p2q}
 \end{figure}

 \begin{theorem}
  For prime numbers $ p>q>2 $ and $ i\geqslant 2 $,
\begin{enumerate} 
\item[(i)]  If $ A=\displaystyle\sum_{a+b+c=i,~a,b\geqslant 1}{\phi(p)\choose a}{\phi(q)\choose b}{\phi(p^2)+\phi(pq)\choose c}
+ \displaystyle\sum_{a+b=i-\phi(p^2),~a\geqslant 1}{\phi(p)\choose a}{\phi(pq)\choose b}
+ \displaystyle\sum_{a+b=i-\phi(pq),~a\geqslant 1}{\phi(q)\choose a}{\phi(p^2)\choose b}
$, then 
  \[
   d(\Gamma(\mathbb{Z}_{p^2q}),i)=
   \left\{
\begin{array}{lr}
 A+1
\quad\mbox{if ~~$i=\phi(p^2)+\phi(pq)$,}\\[15pt]
A
\quad\mbox{~~~~~otherwise.}\\ 
 \end{array}
\right.
\]

 \item[(ii)]
 $D_t(\Gamma(\mathbb{Z}_{p^2q}),x)=\displaystyle\sum_{i=2}^{\mid V\mid}\displaystyle\sum_{a+b+c=i,~a,b\geqslant 1}{\phi(p)\choose a}{\phi(q)\choose b}{\phi(p^2)+\phi(pq)\choose c}x^i $.
 \end{enumerate}
 \end{theorem}
 
 \begin{proof}
 	\begin{enumerate} 
 		\item[(i)]
 		By Corollary \ref{induced}, the vertex set of $ \Gamma(\mathbb{Z}_{p^2q}) $ (the graph $ G=\Gamma(\mathbb{Z}_{p^2q}) $ have shown in Figure \ref{Z_p2q}) can be partitioned into four disjoint subsets such as
  \[
 V_p=\lbrace px : x = 1, 2,..., pq-1, p\nmid x,q \nmid x \rbrace,~~~~\Gamma(V_p)=\overline{K}_{\phi(pq)},
 \]
  \[
 V_q=\lbrace qx : x = 1, 2,..., p^2-1, p\nmid x \rbrace,~~~~~~~~~~~~\Gamma(V_q)=\overline{K}_{\phi(p^2)},
 \]
  \[
 V_{p^2}=\lbrace p^2x : x = 1, 2,..., q-1 \rbrace,~~~~~~~~~~~~~~~~~~\Gamma(V_{p^2})=\overline{K}_{\phi(q)},
 \]
 \[
 V_{pq}=\lbrace pqx : x = 1, 2,..., p-1 \rbrace,~~~~~~~~~~~~~~~~~~\Gamma(V_{pq})=K_{\phi(p)},
 \]
 Since $ \phi(p^2)>\phi(pq)>\phi(p)>\phi(q)>1 $, so $ \gamma(G)=\gamma_t(G)=2 $ and $ d(G,2) =d_t(G,2)=\phi(p)\phi(q)$.

 For $ i>2 $, to choose a dominating set of size $ i $ there are four following cases:
 
Case 1: If $ x\in V_{pq} , y\in V_{p^2}$, then 
\[d(G,i)=\sum_{a+b+c=i,~a,b\geqslant 1}{\phi(p)\choose a}{\phi(q)\choose b}{\phi(p^2)+\phi(pq)\choose c},
 \]
 
 Case 2: 
If $x\in V_{pq} , y\notin V_{p^2}$, then 
\[
d(G,i)=\sum_{a+b=i-\phi(p^2),~a\geqslant 1}{\phi(p)\choose a}{\phi(pq)\choose b}.
\]

  Case 3: 
If $x\notin V_{pq} , y\in V_{p^2}$, then 
\[
d(G,i)=\sum_{a+b=i-\phi(pq),~a\geqslant 1}{\phi(q)\choose a}{\phi(p^2)\choose b}.
\]
 
 Case 4:  
If  $ x\notin V_{pq} , y\notin V_{p^2}$, then 
\[
d(G,i)={\phi(pq)\choose \phi(pq)}{\phi(p^2)\choose \phi(p^2)}=1.
\]  

 Therefore by the addition principle we have the result.

\item[(ii)]
Every total dominating set contains some vertices of both $ \overline{K}_{\phi(q)} $ and $ K_{\phi(p)} $. So for $ i\geqslant 2 $, $ d_t(G,i)=\sum_{a+b+c=i,~a,b\geqslant 1}{\phi(p)\choose a}{\phi(q)\choose b}{\phi(p^2)+\phi(pq)\choose c} $.
\qed
\end{enumerate} 
 \end{proof}

 \begin{figure}[ht]
 	\centering
 	\includegraphics[scale=0.7]{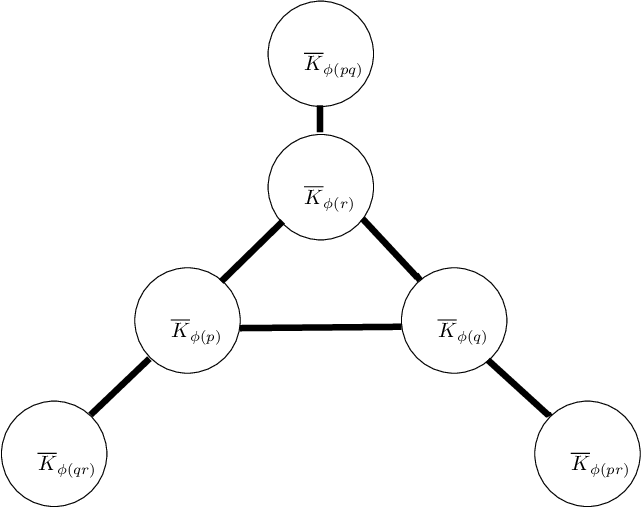}
 	\caption{The graph $\Gamma(\mathbb{Z}_{pqr})$}\label{Z_pqr}
 \end{figure}

 The following theorem gives the number of dominating sets and total dominating sets of $ \Gamma(\mathbb{Z}_{pqr}) $ for prime numbers $ p>q>r>2 $. 
 
 \begin{theorem}
 Let $ \Gamma(\mathbb{Z}_{pqr}) $ be a zero-divisor graph and $ p>q>r>2 $ be distinct prime numbers. Then for $ i\geqslant 3 $,
 \begin{enumerate} 
\item[(i)] 
   $   d(\Gamma(\mathbb{Z}_{pqr}),i)=
   \left\{
\begin{array}{lr}
 A+B+C+1
\quad\mbox{if ~~$i=\phi(pq)+\phi(pr)+\phi(qr)$,}\\[15pt]
A+B+C
\quad\mbox{~~~~otherwise.}\\ 
 \end{array}
\right.
 $ 
 
 where;
\[ A=\sum_{a+b+c+d=i,~a,b,c\geqslant 1}{\phi(p)\choose a}{\phi(q)\choose b}{\phi(r)\choose c}{\phi(pr)+\phi(pq)+\phi(qr)\choose d},\]
\[ B=\sum_{a+b+c=i-\phi(pr),~a,b\geqslant 1}{\phi(p)\choose a}{\phi(r)\choose b}{\phi(pq)+\phi(qr)\choose c},\]
\[ C=\sum_{a+b=i-(\phi(pr)+\phi(qr)),~a\geqslant 1}{\phi(r)\choose a}{\phi(pq)\choose b},
\] 
likewise $ B,C $ is computed for two other cases.
 \item[(ii)]
 $D_t(\Gamma(\mathbb{Z}_{pqr}),x)=\displaystyle\sum_{i=3}^{\mid V\mid}\displaystyle\sum_{a+b+c+d=i,~a,b,c\geqslant 1}{\phi(p)\choose a}{\phi(q)\choose b}{\phi(r)\choose c}{\phi(pr)+\phi(pq)+\phi(qr)\choose d}x^i $.
 \end{enumerate}
 \end{theorem}
 \begin{proof}
 	\begin{enumerate} 
 		\item[(i)] 
 By Corollary \ref{induced}, the vertex set of $ \Gamma(\mathbb{Z}_{pqr}) $ can be partitioned into six disjoint subsets such as 
 \[
 V_p=\lbrace px : x = 1, 2,..., qr-1, q\nmid x,r \nmid x \rbrace,~~~~\Gamma(V_p)=\overline{K}_{\phi(qr)},
 \]
 \[
 V_q=\lbrace px : x = 1, 2,..., pr-1, p\nmid x,r \nmid x \rbrace,~~~~\Gamma(V_q)=\overline{K}_{\phi(pr)},
 \]
 \[
 V_r=\lbrace px : x = 1, 2,..., pq-1, p\nmid x,q \nmid x \rbrace,~~~~\Gamma(V_r)=\overline{K}_{\phi(pq)},
 \]
  \[
 V_{pq}=\lbrace pqx : x = 1, 2,..., r-1\rbrace,~~~~~~~~~~~~~~~~~\Gamma(V_{pq})=\overline{K}_{\phi(r)},
 \] 
 \[
 V_{pr}=\lbrace prx : x = 1, 2,..., q-1\rbrace,~~~~~~~~~~~~~~~~~\Gamma(V_{pr})=\overline{K}_{\phi(q)},
 \]
  \[
 V_{qr}=\lbrace qrx : x = 1, 2,..., p-1\rbrace,~~~~~~~~~~~~~~~~~\Gamma(V_{qr})=\overline{K}_{\phi(p)},
 \]
 
 The graph $ G=\Gamma(\mathbb{Z}_{pqr}) $ have shown in Figure \ref{Z_pqr}.  Since $ \phi(pq)>\phi(pr)>\phi(qr)>\phi(p)>\phi(q)>\phi(r)>1 $, so $ \gamma(G)=\gamma_t(G)=3 $ and $ d(G,3) =d_t(G,3)=\phi(p)\phi(q)\phi(r)$.

 For $ i>3 $, to choose a dominating set of size $ i $ there are eight following cases:
  
  Case 1: If $x\in V_{pq}, y\in V_{pr}, z\in V_{qr}$, then 
    \[
  d(G,i)=\sum_{a+b+c+d=i,~a,b,c\geqslant 1}{\phi(p)\choose a}{\phi(q)\choose b}{\phi(r)\choose c}{\phi(pq)+\phi(pr)+\phi(qr)\choose d}.
  \] 
  
Case 2: If $x\notin V_{pq}, y\in V_{pr}, z\in V_{qr}$,  then

\[
d(G,i)= \sum_{a+b+c=i-\phi(pq),~a,b\geqslant 1}{\phi(p)\choose a}{\phi(q)\choose b}{\phi(pr)+\phi(qr)\choose c}.
\]

 Case 3: If $x\in V_{pq}, y\notin V_{pr}, z\in V_{qr}$, then
 \[
 d(G,i)=\sum_{a+b+c=i-\phi(pr),~a,b\geqslant 1}{\phi(p)\choose a}{\phi(r)\choose b}{\phi(pq)+\phi(qr)\choose c}.
 \] 
 
 Case 4: If $x\in V_{pq}, y\in V_{pr}, z\notin V_{qr}$, then 
 \[
 d(G,i)=\sum_{a+b+c=i-\phi(qr),~a,b\geqslant 1}{\phi(q)\choose a}{\phi(r)\choose b}{\phi(pq)+\phi(pr)\choose c}.
 \]
  
  Case 5: If $x\notin V_{pq} , y\notin V_{pr}, z\in V_{qr}$, then 
  \[
  d(G,i)=\sum_{a+b=i-(\phi(pq)+\phi(pr)),~a\geqslant 1}{\phi(p)\choose a}{\phi(qr)\choose b}. 
  \]
  
  Case 6: 
 If $x\notin V_{pq}, y\in V_{pr}, z\notin V_{qr},$ then 
 \[
 d(G,i)=\sum_{a+b=i-(\phi(pq)+\phi(qr)),~a\geqslant 1}{\phi(q)\choose a}{\phi(pr)\choose b}. 
 \]
 
 Case 7: 
  If $x\in V_{pq}, y\notin V_{pr}, z\notin V_{qr}$, then 
  \[
  d(G,i)=\sum_{a+b=i-(\phi(pr)+\phi(qr)),~a\geqslant 1}{\phi(r)\choose a}{\phi(pq)\choose b}.
  \] 
  
  Case 8: 
  If  $x\notin V_{pq}, y\notin V_{pr}, z\notin V_{qr}$, then 
  \[
  d(G,i)={\phi(pq)\choose \phi(pq)}{\phi(pr)\choose \phi(pr)}{\phi(qr)\choose \phi(qr)}=1.
  \] 
  
 Therefore the result follows.
 
 \item[(ii)] 
  Every total dominating set contains some vertices of  $\overline{K}_{\phi(p)}$, $ \overline{K}_{\phi(q)}$ and $ \overline{K}_{\phi(r)}$. So for $ i\geqslant 3 $,  
 \[
  d_t(G,i)=\displaystyle\sum_{a+b+c+d=i,~a,b,c\geqslant 1}{\phi(p)\choose a}{\phi(q)\choose b}{\phi(r)\choose c}{\phi(pq)+\phi(pr)+\phi(qr)\choose d}.
  \]
  \qed
  \end{enumerate} 
 \end{proof}

  \begin{figure}[h]
  	\begin{minipage}{7.5cm}
  		\includegraphics[width=\textwidth]{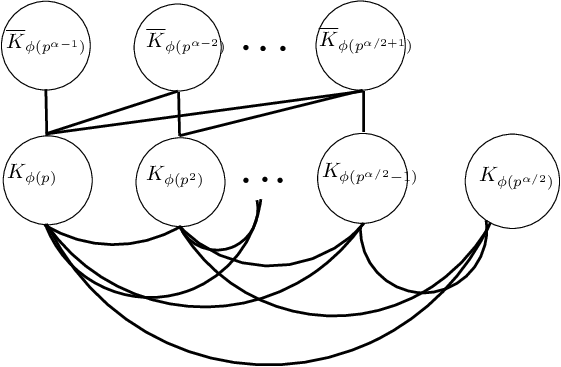}
  	\end{minipage}
  	\hspace{1cm}
  	\begin{minipage}{7.5cm}
  		\includegraphics[width=\textwidth]{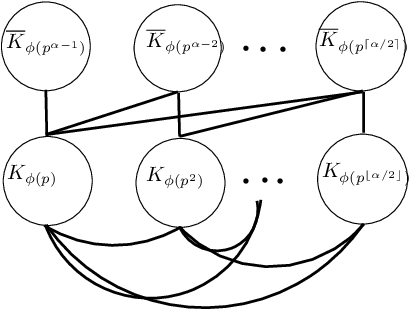}
  	\end{minipage}
  	\caption{The graph $\Gamma(\mathbb{Z}_{p^{\alpha}})$ for an even integer $ \alpha $ (left) and an odd integer $ \alpha $ (right).}\label{Z_pa}
  \end{figure}

Here, we obtain the number of the dominating sets of $\Gamma(\mathbb{Z}_{p^{\alpha}})$. 
 
\begin{theorem}
 Suppose that $ p>2 $ is a prime number. 
  \begin{enumerate}
  \item[(i)]
  If $\alpha>2$ is an odd integer then,
  \[
   d(\Gamma(\mathbb{Z}_{p^{\alpha}}),i)=
   \left\{
\begin{array}{lr}
 A+1
\quad\mbox{if ~~$i=\displaystyle\sum_{k=1}^{\lfloor\alpha/2\rfloor}\phi(p^{\alpha-k})$,}\\[15pt]
A 
\quad\mbox{~~~~otherwise.}\\ 
 \end{array}
\right.
 \]
where; $
A=\displaystyle\sum_{1\leqslant r\leqslant\lfloor\alpha/2\rfloor, a\geqslant 1}{\phi(p^r)\choose a}{n-\phi(p^r)-\displaystyle\sum_{k=1}^{r-1}\phi(p^k)+\phi(p^{\alpha-k})\choose i-a-\displaystyle\sum_{k=1}^{r-1}\phi(p^{\alpha-k})}
 $.
  \item[(ii)]
  If $\alpha>2$ is an even integer then,
  \[
   d(\Gamma(\mathbb{Z}_{p^{\alpha}}),i)= 
   \left\{
\begin{array}{lr}
 A 
\quad\mbox{if ~~$i\leqslant\displaystyle\sum_{k=1}^{(\alpha/2)-1}\phi(p^{\alpha-k})  $,}\\[15pt]
A+B 
\quad\mbox{~~~~otherwise.}\\ 
 \end{array}
\right.
 \]
 where;

\[
A=\displaystyle\sum_{1\leqslant r<\alpha/2, a\geqslant 1}{\phi(p^r)\choose a}{n-\phi(p^r)-\displaystyle\sum_{k=1}^{r-1}\phi(p^k)+\phi(p^{\alpha-k})\choose i-a-\displaystyle\sum_{k=1}^{r-1}\phi(p^{\alpha-k})} 
\] and

\[ 
B=\displaystyle\sum_{ i=a+\sum_{k=1}^{(\alpha/2)-1}\phi(p^{\alpha-k}) ,~a\geqslant 1}{\phi(p^{\alpha/2})\choose a}.
\]
  \end{enumerate} 
 \end{theorem}
 
 \begin{proof}
 	  \begin{enumerate} 
\item[(i)] 
For an integer $ \alpha>2 $, the vertex set of $ \Gamma(\mathbb{Z}_{p^{\alpha}}) $ is the disjoint union of the following  sets $ V_{p^{1}},...,V_{p^{\lceil\alpha/2\rceil}},...,V_{p^{\alpha -1}} $. 
 \[
 V_{p^{1}}=\lbrace px : x=1,...,p^{\alpha -1}-1 , p\nmid x\rbrace,~~~~~~~~~~~~~~~~~~~~~~~~~~~\Gamma(V_p)=\overline{K}_{\phi(p^{\alpha -1})},
 \]
 \[
  V_{p^{2}}=\lbrace p^2x : x=1,...,p^{\alpha -2}-1 , p\nmid x\rbrace,~~~~~~~~~~~~~~~~~~~~~~~~~\Gamma(V_{p^2})=\overline{K}_{\phi(p^{\alpha -2})},
 \]
 \[
 \vdots~~~~~~~~~~~~~~~~~~~~~~~~~~~~~~~~~~~~~~~~~~~~~~~~~~~~~~~~~~~~~~~\vdots
 \]
  \[
 V_{p^{\lceil\alpha/2\rceil}}=\lbrace p^{\lceil\alpha/2\rceil}x : x=1,...,p^{\alpha-\lceil\alpha/2\rceil}-1 , p\nmid x\rbrace,~~~~~~\Gamma(V_{p^{\lceil\alpha/2\rceil}})=K_{\phi(p^{\lfloor\alpha/2\rfloor})},
 \]
  \[
 \vdots~~~~~~~~~~~~~~~~~~~~~~~~~~~~~~~~~~~~~~~~~~~~~~~~~~~~~~~~~~~~~~~\vdots
 \]
 \[
 V_{p^{\alpha -1}}=\lbrace p^{\alpha -1}x : x=1,...,p-1 \rbrace,~~~~~~~~~~~~~~~~~~~~~~~~~~~~~~~\Gamma(V_{p^{\alpha -1}})=K_{\phi(p^1)}.
 \]
 
 Since any vertex $ v\in V_{p^{\alpha -1}} $ is adjacent to all vertices of $\Gamma(\mathbb{Z}_{p^{\alpha}}) $, this implies that each vertex of $ V_{p^{\alpha-1}} $ forms an dominating set, so $\gamma( \Gamma(\mathbb{Z}_{p^{\alpha}}) )=1$ and $\gamma_t( \Gamma(\mathbb{Z}_{p^{\alpha}}) )=2$. So  we have $ d( \Gamma(\mathbb{Z}_{p^{\alpha}}),1)=\phi(p)  $ and 
 $ d_t( \Gamma(\mathbb{Z}_{p^{\alpha}}),2)={\phi(p)\choose 2}+\phi(p)(n-\phi(p))   $, where $ n=\mid V(\Gamma(\mathbb{Z}_{p^{\alpha}}))\mid $.

 The graph $ G=\Gamma(\mathbb{Z}_{p^{\alpha}}) $ is obtained by Corollary \ref{induced}, which have shown in Figure \ref{Z_pa}. 
First assume that $ \alpha>2 $ is an odd integer. For $  i > 1 $, to choose a dominating set of size $ i $ there are the following cases: 

Case 1: 
If $x\in V_{p^{\alpha -1}}$, then $d(G,i)=\sum_{a+b=i ,~a\geqslant 1}{\phi(p)\choose a}{n-\phi(p)\choose b}$. 

Case 2: 
If $x\in V_{p^{\alpha -2}}, x\notin V_{p^{\alpha -1}}$ then 

\[d(G,i)=\sum_{a+b=i-\phi(p^{\alpha-1}) ,~a\geqslant 1}{\phi(p^2)\choose a}{n-[(\phi(p)+\phi(p^2)+\phi(p^{\alpha-1})]\choose b}.
\]
\[
\vdots
\]

Case $\lfloor\frac{k}{2}\rfloor$: 
If $x\in V_{p^{\alpha -\lfloor\alpha/2 \rfloor }}, x\notin V_{p^{\alpha -k}}, 1\leqslant k<\lfloor\alpha/2 \rfloor$, then
\[
d(G,i)=\sum_{a+b=i-\sum_k\phi(p^{\alpha-k}) ,~a\geqslant 1}{\phi(p^{\lfloor\alpha/2\rfloor})\choose a}{\phi(p^{\lceil \alpha/2\rceil})\choose b}.
\]

Case k: 
If $x\notin V_{p^{\alpha -k}}, 1\leqslant k\leqslant\lfloor\alpha/2 \rfloor, i=\sum_k\phi(p^{\alpha-k})$   then  $d(G,i)=1$ and so,
\[
A=\sum_{1\leqslant r\leqslant\lfloor\alpha/2\rfloor, a\geqslant 1}{\phi(p^r)\choose a}{n-\phi(p^r)-\sum_{k=1}^{r-1}\phi(p^k)+\phi(p^{\alpha-k})\choose i-a-\sum_{k=1}^{r-1}\phi(p^{\alpha-k})}
\]
By addition principle we have the result.

\item[(ii)] 
 There are the same conditions of Part (i) to choose dominating set of size $ i>2 $ except two last cases. So,
\[
A=\sum_{1\leqslant r<\alpha/2, a\geqslant 1}{\phi(p^r)\choose a}{n-\phi(p^r)-\sum_{k=1}^{r-1}\phi(p^k)+\phi(p^{\alpha-k})\choose i-a-\sum_{k=1}^{r-1}\phi(p^{\alpha-k})} 
,\]

If $x\in V_{p^{\alpha/2}}, x\notin V_{p^{\alpha -k}}, 1\leqslant k<\alpha/2$  then $B=\sum_{ i=a+\sum_k\phi(p^{\alpha-k}) ,~a\geqslant 1}{\phi(p^{\alpha/2})\choose a}.$
  \qed
  \end{enumerate} 
   \end{proof}
 
\begin{theorem}
Let $ p>2 $ be a prime number. For an integer $ \alpha>2 $, the number of the  total dominating sets of graph $ \Gamma(\mathbb{Z}_{p^{\alpha}}) $ of order $ n $ is
\[
d_t(\Gamma(\mathbb{Z}_{p^{\alpha}}),i)=\sum_{a+b=i ,~a,b\geqslant 1}{\phi(p)\choose a}{n-\phi(p)\choose b}
\]
\end{theorem}
 \begin{proof}
 Since  every vertex of null subgraph $ \Gamma(V_p) $ is adjacent to any vertex of complete subgraph $ \Gamma(V_{p^{\alpha-1}}) $ exactly, so any total dominating set of size $ i\geqslant2 $ of the graph $ \Gamma(\mathbb{Z}_{p^{\alpha}}) $ must containe some vertex of complete subgraph $ \Gamma(V_{p^{\alpha-1}}) $.
  \qed
 \end{proof}


\end{document}